\documentclass[11pt,reqno]{amsart}
\usepackage{amssymb,amsmath}
\usepackage[bbgreekl]{mathbbol}
\usepackage[toc,page]{appendix}
\usepackage{mathrsfs}
\usepackage[normalem]{ulem} 

\usepackage{xcolor}
\usepackage[all]{xy}
\newcommand{\calf}{{\mathcal F}}

\newcommand{\aut}{{\rm Aut}}

\newcommand{\der}{\operatorname{Der}}

\newcommand{\Id}{{\rm Id}}

\newcommand{\spec}{\operatorname{Spec}}

\newcommand{\A}{{\mathbb A}}

\newcommand{\C}{{\mathbb C}}

\newcommand{\K}{{\Bbbk}}
\newcommand{\N}{{\mathbb N}}

\newcommand{\Z}{{\mathbb Z}}
\newcommand{\G}{{\mathbb G}}
\newcommand{\kK}{{\mathbb k}}
\newcommand{\kk}{{\mathbb k}}

\newcommand{\pol}{{\Bbbk[x_1,\dots,x_n]}}

\newtheorem{thm}{Theorem}[section]
\newtheorem{pro}[thm]{Proposition}
\newtheorem{cor}[thm]{Corollary}
\newtheorem{lem}[thm]{Lemma}

\theoremstyle{definition}
\newtheorem{defi}[thm]{Definition}
\newtheorem{quest}[thm]{Question}
\newtheorem{rem}[thm]{Remark}

\newtheorem{exa}[thm]{Example}

\usepackage{enumitem}

\newcommand{\beginenumit}{\begin{enumerate}[leftmargin=\parindent,align=left,labelwidth=\parindent,parsep=8pt,labelsep=0pt,label=\textbullet\ ]}
\newcommand{\beginenum}{\begin{enumerate}[leftmargin=\parindent,align=left,labelwidth=\parindent,parsep=8pt,labelsep=0pt,label=(\arabic*)\ ]}
\newcommand{\beginenumr}{\begin{enumerate}[leftmargin=\parindent,align=left,labelwidth=\parindent,labelsep=0pt,parsep=8pt,label=(\roman*)\ ]}
\newcommand{\beginenuma}{\begin{enumerate}[leftmargin=\parindent,align=left,labelwidth=\parindent,labelsep=0pt,parsep=8pt,label=(\alph*)\ ]}

\begin{document}

\title[]{On polynomial automorphisms commuting with a simple derivation}

\maketitle
\begin{center}
  {\sc Pierre-Louis Montagard}\footnote{\label{imag}IMAG, Univ Montpellier, CNRS, Montpellier, France},  
{\sc Iván Pan}\footnote{\label{cmat}CMAT, Universidad de la Republica, Montevideo, Uruguay, supported by CSIC (Udelar), ANII and PEDECIBA, of Uruguay}, {\sc Alvaro Rittatore}\footref{cmat}\let\thefootnote\relax\footnotetext{This work was partially supported by the CNRS International Research Laboratory IFUMI}
\end{center}

 \begin{abstract}
Let $D$ be a simple  derivation of the
polynomial ring $\pol$, where $\kk$ is an algebraically closed field of characteristic zero, and denote by $\aut(D)\subset\aut\bigl(\pol\bigr)$ the subgroup of $\kk$-automorphisms commuting with $D$.  We show that the connected component of $\aut(D)$  passing through the identity is a
unipotent algebraic group of dimension at most $n-2$, this bound being
sharp. Moreover, $\aut(D)$ is an algebraic group  if and only
if it is a connected ind-group.   Given a simple derivation $D$, we characterize when $\aut(D)$ contains a normal subgroup of translations. As an application of our techniques we
show that if $n=3$, then  either $\aut(D)$  is a discrete group or it is isomorphic to
the additive group acting by translations,  and give some insight on the case $n=4$.
 \end{abstract}

\section{Introduction}

Let $\pol$ be the polynomial ring over an algebraically closed field
$\K$ of characteristic 0. Recall that  a
$\K$-derivation of $\pol$ is a linear endomorphism  $D\in \operatorname{End}\bigl(\pol\bigr)$
such that $D(fg)=D(f)g+fD(g)$ for all $f,g\in \pol$; we denote $D\in
\der\bigl(\pol\bigr)$.

A central and, of course, very difficult problem is to classify
derivations up  conjugation by an automorphism of 
$\pol$.  If $n=2$ there is a partial
classification (\cite{BP2}), but essentially nothing is known in
higher dimension. In order to provide contributions to  the solution of this
 problem, one may try to classify derivations whose interest has
been proven in relation to various areas of mathematics. This is the
case, for example, of the  \emph{locally nilpotent
  derivations} and the  
\emph{simple derivations} (see Definition \ref{defi:locsimder}  below). In the first case, significant progress has
been made in recent decades (see \cite{kn:Freudenburg} and references
therein), while little is known about the second one.

Recall that if  $\Delta$ is a  locally nilpotent
  derivation  then  the formal series
  $e^{t\Delta}:=\sum_{i=0}^\infty
  \frac{t^i\Delta^i}{i!}$ 
acts as a finite sum on every polynomial and defines an element in
$\aut\bigl(\pol\bigr)$ (the group of $\K$-automorphisms of $\pol$) for all
$t\in \K$; in this way, $\Delta$ induces   an action of the 
additive group $\G_a=(\K,+)$ on $\pol$.
 Conversely, every action of
$\G_a$ on $\pol$ is of this form (see for example \cite[\S 1.5]{kn:Freudenburg}). When
$\K=\C$, the field of complex numbers,  if one associates to a
derivation $D$ the  polynomial vector field $\bigl(D(x_1), 
\ldots, D(x_n)\bigr)$, then $D$ is a locally 
nilpotent derivation if and only if the   general solution  of the differential
equation of the corresponding vector field is  polynomial (\cite{FiWa}).    

On the other hand, given a   simple derivation  $D\in\der\bigl(\pol\bigr)$  --- that is the only $D$-stable ideals are the trivial ones
---, one may extend  the polynomial ring $\pol$ to a ``skew''
polynomial ring $R$ by adjoining a new
indeterminate $t$ with the rule: $t\cdot f-f\cdot t =D(f)$ for any
$f\in\pol$. It is well known that $D$ is simple if and only if
$R$ has no non trivial bilateral ideals 
(see for example \cite[Chap. 2]{GoWa}). Finally, if $\K=\C$ and we consider the singular
foliation $\calf$ associated with the vector field $\bigl(D(x_1), \ldots,
D(x_n)\bigr)$ on $\C^n$, then $D$ being simple is equivalent to saying that
$\calf$ is not singular and any of its leaves is Zariski dense. For a more
general field, the previous  equivalence may be
described in terms of algebraic 
independence of power series in one indeterminate (\cite{Le}). 

Thus, if we compare the two kind of derivations previously
mentioned by focusing on
their corresponding flows (when $\K=\C$), then  locally
nilpotent and simple derivations  appear as opposite in a certain
sense.  More generally, if $\K$ is an arbitrary field and $\Delta$ is a 
locally nilpotent derivation of $\pol$, then  $\ker \Delta=\bigl\{f\in\pol; \Delta(f)=0\bigr\}$ is a $\K$-subalgebra of $\pol$
with transcendence degree over $\K$ equal to $n-1$
(see for example \cite[Pro.1.3.32(i)]{vdE}), whereas clearly $\ker
D=\K$ when $D$ is a
simple derivation.  One finds further evidence for this intuition  on the opposite behavior of simple and locally nilpotent derivations in
  dimension $2$  (\cite{MePa}, \cite{Pa}), looking at the automorphisms
  group of the derivations:

  \begin{defi}\label{defiautD}
Let $D\in \der\bigl(\pol\bigr)$. The \emph{automorphisms group of $D$}
is the isotropy group of $D$ for the action of $\aut\bigl(\pol\bigr)$
on $\der\bigl(\pol\bigr)$
by conjugations:
\[
  \aut(D)= \aut\bigl(\pol\bigr)_D\subset \aut\bigl(\pol\bigr).
  \]
  \end{defi}

\begin{thm}\label{th:MePan} If $D\in\der\bigl(\K[x_1,x_2]\bigr)$, then   

$(a)$ $D$ is locally nilpotent if and only if $\aut(D)$ is not an algebraic group.

$(b)$  If $D$ is simple, then $\aut(D)=1$.  \hfill \qed
\end{thm}

Recall that if $n\geq 2$, then  $\aut\bigl(\pol\bigr)$ does not admit a compatible
structure of group scheme, but it is an  ind-group  with filtration
induced by total degree (see for example \cite{Sha} and \cite[Chapter 4]{Ku}). More  precisely, if
we consider the set theoretical inclusion
$\aut\bigl(\pol\bigr)\hookrightarrow \pol^n$, given by $\varphi\mapsto
\bigl(\varphi(x_1),\dots, \varphi(x_n)\bigr)$, then the \emph{degree of
$\varphi$} induces a filtration 
 $\aut\bigl(\pol\bigr)=\bigcup_p \aut\bigl(\pol\bigr)_p$, where
  \[
\aut\bigl(\pol\bigr)_p= \bigl\{\varphi=(f_1,\ldots,f_n)\in\aut\bigl(\pol\bigr) \mathrel{ : } \forall i\  \operatorname{deg}(f_i)\leq p\ \bigr\}.
\]

It  follows that  a closed group $G\subset
\aut\bigl(\pol\bigr)$ is algebraic if and only if there exists $p$
such that $G\subset
\aut\bigl(\pol\bigr)_p$.
Clearly, $\aut\bigl(\pol\bigr)$ identifies in a canonical way
with $\aut(\A^n)$, the automorphisms group of the $n$-dimensional affine
space;  under this identification,  if $G\subset \aut\bigl(\pol\bigr)$ is algebraic, then the
induced  action $G\times \A^n\to\A^n$ is regular.

If $n>2$, it is not difficult to show that $\aut(D)$ is never
algebraic when $D$ is locally nilpotent; however, very few is known about
the simple case. In fact,   $\aut(D)$ is expected to be
  algebraic:  in \cite{Ya}, after exhibiting examples of
simple derivations whose isotropy is $\G_a$ acting  by translations,
the author conjectures that in general, up to conjugation, if $D$ is
simple then $\aut(D)$ acts by translations.

The main objective of
this work is to give evidence in the direction 
of proving that $\aut(D)$ is in fact an algebraic group.
 The structure of this work is as follows:

In Section \ref{sec:prelim} we collect some basic definitions and
  results on $\der\bigl(\pol\bigr)$ and the ind-group structure of $\aut
\bigl( \pol  \bigr)$.

In Section \ref{sec:autsimpder} we  prove that if $D$ is a simple
derivation, then $\aut(D)^0$, the connected component of $\aut(D)$
passing through the identity, is a unipotent algebraic group (Theorem
\ref{thm:aut0conn}), and that any algebraic element of $\aut(D)$ is unipotent
(Theorem \ref{thm:aut0unip}). In
particular, $\aut(D)$ is algebraic if and only if it is a connected group.

In Section \ref{sec:simpleforunip} we show that if
$D\in\operatorname{Der}\bigl(\pol\bigr)$ is a simple 
derivation, then $\dim\aut(D)^0\leq n-2$ (Theorem \ref{thm:dimaut0}), and that this
bound is sharp (Corollary \ref{cor:boundsharp}).

When a simple derivation is such that $\aut(D)$ contains a non trivial
translation, then one can give more insight on the structure of
$\aut(D)^0$; this is done in Section \ref{sec:translations}. 
More precisely, if $D$ is a simple derivation, we characterize in terms of a coordinate system when $\aut(D)$ contains a  non trivial normal subgroup of translations  and we use this characterization to somewhat describe  $\aut(D)^0$ (Theorem \ref{th:AutAlg}). We apply our point
of view  in order to prove that  the  isotropy group of a simple
derivation of $\K[x_1,x_2,x_3]$ is either trivial, either the 
additive group acting by translations (upon conjugation) or  it is
an infinite  countable discrete
group  ---
however, we remark that, up to our knowledge, there is
no known example of this last possibility.

Most examples of explicit simple derivations  of $\pol$ are
 produced by extending, recursively, a simple derivation in $\K[x_1]$,
 using Shamsuddin's criterion (see Proposition \ref{prop:sham}), or some of
 its variants, and are therefore of the form $D=\partial/\partial
 x_1+\sum_{i=2}^na_i \partial/\partial x_i$, $a_i\in \pol$. In this case --- we say that $D$  admits $x_1$ as a \emph{linear coordinate} ---, one can describe quite well the action of an element of $\aut(D)^ 0$  on the coordinate $x_1$ (see Lemma \ref{lem:deltatrans}). If $n=4$,  we use this description in order to give more insight on the algebraicity of $\aut(D)$ (Theorem \ref{thm:dim4}).

\section{Preliminaries}
\label{sec:prelim}

In this section we recall some basic definitions and results on simple
and locally nilpotent  derivations.

Let $V$ be a $\Bbbk$-vector space. A linear endomorphism
$\varphi:V\to V$ is \emph{locally finite} if for all $v\in V$, there exists a
finite dimensional $\varphi$-stable vector space $W$ containing
$v$. If moreover $\varphi|_{_W}$ is nilpotent (resp. semisimple) for all finite
dimensional $\varphi$-stable vector space $W$, then $\varphi$ is said
to be \emph{locally nilpotent} (resp. \emph{locally semisimple}). We say that    
$\varphi$ is \emph{locally unipotent} if $\varphi-\Id$
is locally nilpotent.

If $\varphi\in\aut\bigl(\pol\bigr)$, we denote by $\langle\varphi \rangle$ the
group generated by $\varphi$; we say that $\varphi$ is \emph{algebraic} if
 the closure $\overline{\langle\varphi \rangle}\subset
\aut(\A^n)$ is an algebraic group --- along this work, 
  closures are taken with respect to the inductive topology of the
  ind-variety structure. It is well known that $\varphi$ is algebraic
  if and only if the total degree of  the positive powers of 
  $\varphi$ is bounded:  $\max \bigl\{ \deg(\varphi^\ell\bigr)\mathrel{:}
  \ell\geq 0\bigr\}<\infty$ (see  for example \cite[Lemma 9.1.4]{kn:furterkraft}).

  Finally, it is easy to prove that $\varphi$ is algebraic if and only
  if  $\varphi$ is locally
		finite  (see \cite[Lemma 9.1.4]{kn:furterkraft});
in particular, $\varphi$ is an algebraic unipotent automorphism if and
only if $\varphi$ is locally unipotent.

\begin{defi}\label{defi:locsimder}
A  \emph{locally nilpotent derivation}
  is a derivation $D$ that is  locally nilpotent as a linear endomorphism of
   $\pol$.

  Let $D$ be a derivation of $\kk[x_1,\ldots,x_n]$. An ideal $I$ of
$\kk[x_1,\ldots,x_n]$ is \emph{$D$-stable} if $D(I)\subset I$. If the
only $D$-stable  ideals are the trivial ones, we say that $D$ is a
\emph{simple derivation}.
\end{defi}

\begin{rem}\label{rem:restsim}
(1) It is clear that if $D$ is a derivation of $\pol$, then $D=\sum_i a_i\frac{\partial}{\partial x_i}$, with $a_i\in
  \kk[x_{1},\dots, x_n]$.

\noindent (2) Let $D=\sum_{i=1}^n a_i\frac{\partial}{\partial x_i}$
be a derivation of $\pol$. Then the ideal $\langle
a_1,\ldots,a_n\rangle_{\pol}\subset \pol$ is $D$-stable.

If $D$ is a simple derivation then $\langle a_1,\dots,
a_n\rangle=\pol$.

  \noindent (3) If  $D=\sum_{i=1}^n a_i\frac{\partial}{\partial x_i}\in\der\bigl(\pol\bigr)$  
  and $s$ is such that 
  $a_i\in\kk[x_{s+1},\dots, x_n]$  for $i>s$,  then
  the restriction $\overline{D}=D|_{_{\kk[x_{s+1},\dots, x_n]}}:\kk[x_{s+1},\dots,
  x_n]\to \kk[x_{s+1},\dots, x_n]$ is also a derivation. If moreover
  $D$ is simple, then $\overline{D}$ is also a simple derivation.
  \end{rem}

\begin{defi}
	Let $D\in  \der\bigl( \kk[x_1,\dots, x_n]\bigr)$. We say that $f\in
	\pol\setminus \{0\}$ is a \emph{Darboux polynomial for $D$ of
		eigenvalue $\lambda\in \kk[x_1,\dots, x_n]$} if $D(f)=\lambda
	f$. In the literature $f$ is also called an \emph{eigenvector of $D$}. 
\end{defi}

\begin{rem}\label{rem:darboux}
	\beginenum
	\item  If $f\in
	\pol$ is a Darboux polynomial  for $D$, then the ideal $\langle
	f\rangle\subset \kk[x_1,\dots, x_n]$ is $D$-stable. 
	
	\item The kernel of a derivation $D$, denoted as $\ker(D)$,   is the subspace of Darboux 
	polynomials of eigenvalue equal to $0$.
	
	\item  In particular,  if
	$D$ is a simple derivation, the only Darboux polynomials are 
	the constant polynomials, and therefore $\ker(D)=\Bbbk$. 
\end{enumerate}
\end{rem}

\begin{rem}\label{rem:GaandLND}
Let  $\Delta\in \der\bigl(\pol\bigr)$ be  a locally nilpotent
derivation and consider $\G_a= \bigl\{e^{t\Delta}\mathrel{:}
t\in\K\bigr\}$. Recall that  the linear span 
  $\langle \Delta\rangle_\K \subset \operatorname \der\bigl(\pol\bigr)$ is
    naturally identified with the Lie algebra of $\G_a$ and 
   $\ker(\Delta)=\pol^{\G_a}$, the subalgebra of
  $\G_a$-invariants.   
\end{rem}

\begin{defi} \label{defi:coord} If $\Delta\in \der\bigl(\pol\bigr)$ is a locally nilpotent
derivation,  we say that $s\in \pol$ is a \emph{slice} for $\Delta$ if $\Delta(s)=1$.

Let $D\in \der\bigl(\pol\bigr)$ be an arbitrary derivation. We say that $s\in \pol$ is a  \emph{linear coordinate
  for  $D$}  if
  $D(s)=1$ and there
  exist polynomials $s_2,\ldots,s_n$ such that
  $\K[s,s_2,\ldots,s_n]=\pol$:
\[
  D=\frac{\partial}{\partial s}+\sum_{i=2}^na_i \frac{\partial}{\partial s_i} \ ,\ a_i\in\K[s,s_2,\dots,s_n].
\]
\end{defi}

\section{Automorphisms of simple derivations: first results}
\label{sec:autsimpder}

As noted in the introduction, the main object of study of this work is
the isotropy subgroup of a simple derivation (see Definition \ref{defiautD}).

\subsection{On the ind-group structure of the isotropy group of a simple derivation}\ %

Recall that $\aut(\A^n)$ is an ind-group, therefore, $\aut(D)$ being
an isotropy group, it is
also an ind-group.  We begin this section by showing that if $D$ is
simple, then $\aut(D)^0\subset \aut(D)$,
the unique  connected component containing the identity, is an
algebraic group.

First, we recall the following result of Baltazar (see \cite[Proposition 7]{Baltazar}): 

\begin{pro}\label{prop:identity}
Let $D$ be a simple derivation of $\pol$. Then every non trivial element of
$\aut(D)$ admits no fixed point. \hfill \qed
\end{pro}

As a consequence of  Proposition \ref{prop:identity} above and some general results about ind-groups,  we obtain the following description of
$\aut(D)$ as an ind-group.

\begin{thm}\label{thm:aut0conn}
If $D\in\der\bigl(\pol\bigr) $ is simple, then  $\aut(D)^0$ is an
affine algebraic group of dimension at most $n$  which acts freely on $\A^n$, and there exists a
family of automorphisms $(\varphi_i)_{i\in \N}$, with $\varphi_i\in\aut(D)$, such that $\varphi_{0}=\Id$ and 
\begin{equation}
  \label{eq:varphilocfin}
  \aut(D)=\bigcup_{i\in \N} \varphi_i\aut(D)^0.
\end{equation}

In particular $X_n=\bigcup_{i=0}^n\varphi_i\aut(D)^0$, $n\in\N$, is a
filtration of $\aut(D)$ by (affine) algebraic varieties.
\end{thm}
\begin{proof}
The first assertion follows from Proposition \ref{prop:identity}
above  and \cite[Propositions 1.8.3 and 7.1.2]{kn:furterkraft}. By
\cite[proposition 1.7.1 and 2.2.1]{kn:furterkraft} it follows that $\aut(D)$
is a countable disjoint union of connected algebraic varieties:
$\aut(D)=\bigcup_{i\in \N} Y_i$,  such that $Y_{0}=\aut(D)^0$, where
$\aut(D)^0$ is  a connected and normal algebraic subgroup of $\aut(D)$.

If we choose $\varphi_i\in Y_i$, $i\neq {0}$, then $\varphi_i\aut(D)^0= Y_i$, since left
multiplication by $\varphi_i$ is an isomorphism of ind-varieties. It
follows that $\aut(D)=\bigcup_{i\in \N} \varphi_i\aut(D)^0$, where $\varphi_{0}=\Id$.
\end{proof}

   \begin{quest}\label{que:count}
     To our knowledge, there are no known  examples of a simple
     derivation of $\pol$ with non-connected automorphisms group.

     \begin{quote} \emph{Is
     it true that if $D$ is a simple derivation, then $\aut(D)$ is
     connected?}
\end{quote}
   
   In view of Theorem \ref{thm:aut0conn}, this would
     imply that any  simple derivation has algebraic automorphisms group. 
     \end{quest}

\begin{thm}\label{thm:aut0unip}
Let  $D\in\der\bigl(\pol\bigr)$ be a simple derivation and $\varphi\in \aut(D)$. If $\varphi$ is algebraic, then $\varphi$ is unipotent.
  In particular, if $G\subset\aut(D)$ is an algebraic subgroup, then $G$ is a unipotent subgroup and therefore $G$ is a closed connected subgroup of $\aut(D)^0$. 
\end{thm}
\begin{proof}
If $\varphi$ is algebraic, then there exists a Jordan decomposition
$\varphi=\varphi_s\varphi_u$, with $\varphi_s$ semi-simple and $\varphi_u$ unipotent and  such that $\varphi_s,\varphi_u\in
\overline{\langle \varphi\rangle}\subset \aut(D)$. By a result of Furter and Kraft (\cite[Proposition
15.9.3.]{kn:furterkraft}),
 $\varphi_s$ acts in $\A^n$ with a fixed point, and it follows from Proposition \ref{prop:identity}
that $\varphi_s$ is trivial. Hence,   $\varphi=\varphi_u$ --- that is,
$\varphi$ is unipotent. 

If $G\subset \aut(D)$ is algebraic, then  we deduce that
  $G$ is  unipotent, so it is connected, proving the second assertion.
\end{proof}

\begin{rem}
  If $D$ is a simple derivation and $\varphi\in \aut(D)$ is a locally
finite automorphism, then  applying Theorem \ref{thm:aut0unip} to
$G=\overline{
  \langle\varphi\rangle}\subset \aut(\A^n)$  we deduce that $\varphi\in\aut(D)^0$.
  \end{rem}

As a direct follow up of theorems \ref{thm:aut0unip} and
\ref{thm:aut0conn}, we have the  
following characterization of the algebraicity of the automorphisms
group of a simple derivation.
 
\begin{cor}
   If $D\in\der\bigl(\pol\bigr)$ is  simple  then the following are equivalent:

  \beginenum
\item The ind-group $\aut(D)$ is an algebraic group.

\item The ind-group $\aut(D)$ is connected.

\item Every element $ \varphi\in \aut(D)$ is algebraic.

\item Every element $ \varphi\in \aut(D)$ is a locally finite
  automorphism.

\item Every element $ \varphi\in \aut(D)$ is a unipotent locally finite
  automorphism. \qed
\end{enumerate}
\end{cor}

\subsection{Additive subgroups of $\aut(D)$}\ %

  \begin{defi}\label{defi:globtriv}
Let $X$ be an affine algebraic variety and  $\varphi: \G_a^s\times X\to X$ be a regular action. We  say
that $\varphi$ is \emph{globally (equivariantly) trivial}, if  there
exists an affine algebraic variety $V$ and an action $\phi:\G_a^s\times (
\A^s\times V)$, $\phi\bigl(a,(t,v)\bigr)=(t+a,v)$ for all $a\in \G^s_a$, $t\in
\A^s$, $v\in V$, together with an equivariant isomorphism $X\cong \A^s\times
V$.

We say that $\varphi$ is \emph{locally  trivial} if there exists a cover
  $X=\bigcup_{i=1}^\ell U_i$ by affine open $\G_a^s$-stable subsets
  such that the restrictions $\varphi|_{_{\G_a^s\times
      U_i}}:\G_a^s\times U_i\to U_i$ are
  globally trivial actions.

Let $\G_a^s\hookrightarrow  \aut(\A^n)$ be a closed immersion. We say
that (the image of) $\G_a^s$  \emph{acts by translations} or that
$\G_a$  is
a \emph{group of translations} if the
induced action $\varphi: \G_a^s\times \A^n\to \A^n$ is globally
trivial where  $\A^n\cong \A^{s}\times \A^{n-s}$ and  $\G_a^s$ acts by
translations in the first $s$ coordinates.

A subgroup of automorphisms $\G_a^s\subset \aut\bigl(\pol\bigr)$ is said
to act in a locally trivial (resp. globally trivial, resp. by translations) way
if the induced action $\varphi: \G_a^s\times \A^n\to \A^n$ is so.
\end{defi}

\begin{rem}
  \beginenum
  \item Notice that $\G_a^s$ acts by translations on $\A^n$ if and only if
$\G_a^s$ is conjugated to a subgroup of the group of translations of
$\A^n$ (as subgroups of $\aut(\A^n)$).

 \item  If $\Delta\in \der\bigl(\pol\bigr)$ is a locally nilpotent
  derivation then  $s\in \pol$ is a slice (see Definition
  \ref{defi:coord}) if and  
  only if the canonical action of $H=\bigl\{e^{t\Delta}\mathrel{:}
  t\in\K\bigr\}\cong \G_a$  over $\A^n$ is globally trivial, where 
  $V\cong  \spec\bigl(\ker(\Delta) \bigr)$ --- this is the content of
  the slice theorem, see for example \cite[Corollary 1.26]{kn:Freudenburg}.

If moreover $s$ is a linear coordinate for $\Delta$, then the action
of $H$ is by translations

\item Let  $\varphi: \G_a^s\times X\to X$ be a  globally
  trivial action, with $X\cong \A^s\times V$. Then it is easy to show
  that   the projection  
   $      p_2: X\cong \A^s\times V\to V$ is the geometric quotient. In particular, $ \K[V]=\K[X]^{\G_a^s}$,  the subalgebra of invariants of $\G_a^s$.
\end{enumerate}
  \end{rem}

  In \cite{kn:Freudenburg}, the following characterization of the local
  triviality of a free action of $\G_a$ on $\A^n$ is given.
  
  \begin{defi}
  Let  $\Delta:\pol\to \pol$ be a locally
  nilpotent derivation with kernel $A$. We call  $\operatorname{pl}(\Delta)=A\cap
  \operatorname{Im}(\Delta)$  the  
  \emph{plinth ideal  of
    $\Delta$}.
  \end{defi}

  \begin{lem}\label{lem:FreuTriviality}
Let $\Delta$ be a locally nilpotent derivation of $\pol$ and consider $\G_a=\bigl\{e^{t\Delta}\mathrel{:} 
  t\in \kk\bigr\} \subset
  \aut(D)$  (see Remark \ref{rem:GaandLND}). Then the restricted action
  of $\G_a$ 
  is   locally trivial if and only if $\bigr\langle
  \operatorname{pl}(\Delta)\bigr\rangle_{\pol} =
  \pol$.
    \end{lem}
    \begin{proof}
See \cite[p. 34]{kn:Freudenburg}.
      \end{proof}

If $\Delta$ is a locally nilpotent derivation that commutes with a
simple derivation, the characterization of the local triviality of the
action of  the additive group $\{e^{t\Delta}\mathrel{:} t\in\K\}$
takes the following form:

\begin{pro}\label{prop_locally}
  Let $D\in\der\bigl(\pol\bigr)$ be a simple derivation and assume that there exists a closed
  immersion $
\G_a  \hookrightarrow \aut(D)$. Then the induced action of $\G_a$ on $\A^n$   is
  locally trivial. In particular, the action is proper --- that is,
  the orbit map $\varphi:\G_a\times \A^n\to \A^n\times \A^n$,
  $\varphi(g,x)= \bigr(x,g(x)\bigr)$ is a proper morphism.
  \end{pro}
\begin{proof}
  Let $\Delta 
  \in \der\bigl(\pol\bigr)$ be such that $e^\Delta$ is a generator of (the image of) $\G_a$ and write
  $A=\ker(\Delta)$. Then $e^{t\Delta}D=De^{t\Delta}$ for all $t\in\K$,
  and therefore $[D,\Delta]=0$. Hence,
  $D(A)\subset A$ and it follows that the  plinth ideal
  $\operatorname{pl}(\Delta)\subset 
  A$ is $D$-stable so  $\bigr\langle
  \operatorname{pl}(\Delta)\bigr\rangle_{\pol}$ is also $D$-stable. We
  conclude by Lemma \ref{lem:FreuTriviality} and the simplicity of
  $D$.

  Finally, the   properness of the action follows from \cite{kn:DeFiGe94} (see
  also \cite[Theorem 3.37]{kn:Freudenburg}).
 \end{proof}

We will use   the previous  characterization later in order to study the subgroup of
translations of the isotropy group of a simple derivation, see Lemma \ref{lem:TransDim4}.

\section{Automorphisms of simple derivations : dimension of
  $\aut(D)^0$}
\label{sec:simpleforunip}

We begin this section by dealing with the special case where $D\in\der\bigl(\pol\bigr)$ is simple and $\aut(D)$ contains a subgroup acting in a globally trivial way or by translations.

\begin{pro}\label{prop:tras0}
	Let $D\in\der\bigl(\pol\bigr)$ be a simple derivation and  $H\subset \aut (D)$
	be a (closed) $s$-dimensional subgroup, $H\cong \G_a^s$,  acting in a globally
        trivial way.
        If $B=\pol^H$, then  $\operatorname{Im}D\subset B$.  Moreover, the restriction
        $\overline{D}=D|_{_B}:B\to B$ is a simple derivation, and if
        $\{f_1,\dots , f_\ell\}$ is a generating set of the
        $\K$-algebra $B$, then   $\bigl\langle D(f_1),\dots, D(f_\ell)\bigr\rangle_B= B$ and
$\bigl\langle D(f_1),\dots, D(f_\ell)\bigr\rangle_\pol=\pol$ unless
$B=\K$ --- that is $s=n$. 
      \end{pro}
        \begin{proof}
          By the global triviality of the $\G_a^s$-action, there exists a
          $\G_a^s$-equivariant isomorphism  $\A^n\cong \A^s\times
          \spec(B)$ and therefore $\pol=B[y_1,\dots,y_s]$, with
          $y_1,\ldots, y_s$ algebraically independent over $B$. If 
          $p\in\pol$ and  $\Delta\in \operatorname{Lie}(H)\subset \der\bigl(\pol\bigr)$, then
          $\Delta\bigl(D(p)\bigr)=D\bigl(\Delta(p)\bigr)=0$. We deduce
          that
          \[
            D(p)\in \bigcap_{\Delta\in\operatorname{Lie}(H)}
          \ker(\Delta)=\pol^H=B.  
          \]
Finally, it is clear that the restriction $\overline{D}:B\to B$ must be a simple derivation and,  since  $I=\bigl\langle D(f_1),\dots, D(f_\ell)\bigr\rangle_B$
is  a $\overline{D}$-stable ideal, it is either $0$ or $B$. But
$D(f_i)=0$ implies that $f_i$ is a constant so the result follows.
\end{proof}

\begin{rem}
We will show in Theorem \ref{thm:dimaut0} that, in the notations of Proposition
\ref{prop:tras0},  $s=\dim H\leq n-2$, so  $\pol^H\neq \K$.
\end{rem}

\begin{cor}\label{cor:tras}
	Let $D\in\der\bigl(\pol\bigr)$ be a simple derivation and
        $H\subset \aut (D)$ 
	be a (closed) $s$-dimensional subgroup of translations. Then  there
	exist coordinates such  
	that $D=\sum_i a_i\frac{\partial}{\partial x_i}$, with $a_i\in
	\kk[x_{s+1},\dots, x_n]$.  Moreover, $\bigl\langle a_{s+1},\dots
	,a_n\bigr\rangle_{\Bbbk[x_{s+1},\dots ,x_n]}=\Bbbk[x_{s+1},\dots
	 ,x_n]$ and therefore $\bigl\langle a_{s+1},\dots
	,a_n\bigr\rangle_{\Bbbk[x_{1},\dots ,x_n]}=\Bbbk[x_{1},\dots
	 ,x_n]$.
  \end{cor}
  \begin{proof}
In this case,  $\pol^H$ is a  polynomial ring and result follows directly
from  Proposition \ref{prop:tras0}.
    \end{proof}

 In 
 \cite[theorem 2]{kn:DvEFM},  the authors describe the structure of
 the action of an unipotent group of dimension $n-1$ on an affine
 variety of dimension $n$.   We present here an adaptation to our
 special  case where $X$ is an affine space: 

 \begin{thm}\label{thm:derksenetalorig}
   Let $U$ be a unipotent group  of dimension $n-1$ acting freely on $\A^n$. Then  $\A^{n}$ is $U$-isomorphic to $U\times
  \kk$. \qed 
\end{thm}

Using
Lie-Kolchin Theorem and the well known identification of
$\operatorname{Lie}(U)$ as a sub-Lie algebra of 
$\operatorname{Der}\bigl(\pol\bigr)$  (see Remark
\ref{rem:GaandLND}),  in \cite[proposition 1]{kn:DvEFM} the authors
give the following nice description:  

\begin{pro}\label{pro:derksenetalorig}
Let $U$ be a unipotent group of dimension $n$ acting on $\A^n$ such that there exists $x\in\A^n$ with $U_x=\{\Id\}$. Then, upon
conjugation with  an automorphism,  there exists a basis $\Delta_1,\ldots,\Delta_n$ of the Lie algebra $\operatorname{Lie}(U)$ such that $\Delta_i(x_j)=\delta_{ij}$ for $1\leq i\leq j\leq n$. 
\qed
\end{pro}

\begin{cor}\label{cor:derksenetal}
  Let $U$ be a unipotent group of dimension $d=n-1$ or $d=n$ which acts freely on $\A^n$. Then,  upon
  conjugation with  an automorphism, there exist
  $U$-invariant locally nilpotent 
  derivations 
   $\{\Delta_1,\dots, \Delta_d\}\subset \operatorname{Der}\bigl(\pol\bigr)$  such that
  $\Delta_i(x_j)=\delta_{ij}$ for $i=1,\ldots,d$ and $j=i,\ldots,n$. 
\end{cor}

\begin{proof}
 If $\dim U=n$, the result is a direct consequence of Proposition
 \ref{pro:derksenetalorig}, and if $\dim U=n-1$, then by Theorem
 \ref{thm:derksenetalorig}, the affine space $\A^n$ is $U$-isomorphic
 to $U\times \Bbbk$ and we  conclude by applying again Proposition
 \ref{pro:derksenetalorig} on $U$  acting on itself. 
\end{proof}

 Now we can state the main result of this section.

\begin{thm}
\label{thm:dimaut0}  Let $D$ be a simple derivation of $\pol$, with
$n\geq 2$. Then $\dim \aut(D)^0\leq n-2$. 
\end{thm}
\begin{proof}
By Theorem \ref{thm:aut0conn}, $\aut(D)^0$ is an algebraic
unipotent group of dimension at most $n$. Let $d=\dim\aut(D)^0$ and
suppose that $d=n$ or $d=n-1$. It follows from Corollary
\ref{cor:derksenetal} that, upon a choice of coordinates, there exist
derivations $\{\Delta_1,\ldots,\Delta_d\}$ such that
$\Delta_i(x_j)=\delta_{ij}$ for $i=1,\ldots,d$ and $j=i,\ldots,n$ and
such that the subgroups
$\{e^{t\Delta_i}, t\in \Bbbk\}$  are contained in $\aut(D)^0$.

Let $D=\sum a_i\frac{\partial}{\partial x_i}$. 
By Corollary \ref{cor:tras}, since $e^{t\Delta_1}$ is a translation for all
$t$, we deduce that $a_j\in \Bbbk[x_2,\dots, 
x_{n}]$  for all $j=1,\dots,n$.

 Let  $\overline{D}$ and $\overline{\Delta}_2$ be  the restrictions of
$D$ and  $\Delta_2$  to $\K[x_2,\ldots,x_n]$ respectively.  Then for
$t\in\K$, $e^{t\overline{\Delta_2}}$  is a translation along $x_2$ that
belongs to $\aut(\overline{D})$. We conclude that
$a_j\in\K[x_3,\ldots,x_n]$ for $j=2,\ldots,n$.

By  recurrence, we deduce that $a_{n-1}\in\K[x_n]$; moreover $a_n\in\K$
if $d=n$ or $a_n\in\K[x_n]$ if $d=n-1$. In both cases, we can
consider the restriction $\hat D$ of $D$ to $\K[x_{n-1},x_n]$ which is
simple by Remark \ref{rem:restsim}. Write $\hat
D=a_{n-1}\frac{\partial}{\partial x_{n-1}}+a_n\frac{\partial}{\partial
  x_{n}}$ with $a_{n-1}, a_n\in \K[x_n]$; restricting to the last
coordinate, we  deduce that $a_n\in\K$, and therefore $\hat{D}$ is
not simple (see Remark \ref{rem:exasham}), which is a contradiction.
\end{proof}

Our next goal is to provide examples that show  the bound given by  Theorem \ref{thm:dimaut0}  is 
optimal for all $n\geq 2$, see Example
\ref{exa_maximal} below. In order to do so, we will use the following
criterion, attributed to Shamsuddin.

\begin{pro}\label{prop:sham}
 Let $A$ be a $\Bbbk$-algebra and consider a simple derivation
 $\delta\in\der(A)$. Let $D_{a,b}\in \der\bigl(A[y]\bigr)$  the
 derivation obtained by extending $ \delta$  by $D_{a,b}(y)=ay+b$ with $a,b\in A$. Then $D_{a,b}$ is simple if and only if  
  \[
    \delta h\neq ah+b \text{ for all }h\in A.
  \]
        \end{pro}
\begin{proof}
  See  \cite[Theorem 13.2.1]{Now1}.
\end{proof}

\begin{rem}\label{rem:exasham}
  \beginenum
\item
  Notice that in particular there exist 
        infinitely many simple derivations of the form  $ \frac{\partial}{\partial u}+
        c(u,v)\frac{\partial}{\partial v}$, $c\in \Bbbk[u,v]$.

      \item If $c\in \K[u]$ and $\alpha\in \K^*$, then the  derivation 
    $\delta=\alpha\frac{\partial}{\partial u}+c'(u)\frac{\partial}{\partial
      v}$, is not simple, since $\delta(\alpha^{-1}c)=c'(u)$. 
        \end{enumerate}
      \end{rem}

    \begin{pro}\label{pro:RestricAut}
 Let   $D\in\der\bigl(\Bbbk[u,v,x_1,\dots,
        x_n]\bigr)$ be a simple derivation of the form       $D=\frac{\partial}{\partial u}+
        c(u,v)\frac{\partial}{\partial v}+\sum_{j=1}^n
          b_j(u,v,x_1,\dots,x_n)\frac{\partial}{\partial x_j}$. Then
          $\varphi|_{_{\K[u,v]}}=\Id_{\K[u,v]}$ for all
      $\varphi\in\aut(D)$.
    \end{pro}
    \begin{proof}
Let $\varphi=(f_1,f_2,g_1,\ldots,g_n)\in\aut(D)$. Since $1= \varphi   
    D(u)= D\bigl(\varphi(u)\bigr)=D(f_1)$ we have that  $D(f_1-u)=0$ so $f_1=u+t$ for some
    $t\in\K$.

    We write $f_2=\sum_{r}\alpha_rx^r$, where  
    $r=(r_1,\dots, r_n)\in \Z_{\geq 0}^n$ and $x^r=x_1^{r_1}\dots x_n^{r_n}$  and 
    $\alpha_r\in\K[u,v]$. Let $d$ be the multidegree of $f_2$ for the
    lexicographic order, with $x_1\geq\dots\geq x_n$. First, we show  that
    $d=0$; for this, 
    assume that $d\neq 0$ and  write $c(u,v)=\sum_{k=0}^\ell
    c_k(u)v^k$. Then  $\ell> 0$ (see Remark \ref{rem:exasham}). Let us
    calculate in a explicit way the equality $\varphi(D)(v)=D\varphi(v)$:
    \[
        \varphi D(v) =  \sum_{k=0}^\ell c_k(u+t)f_2^k =
        c_\ell(u+t) \alpha_{r}^\ell x^{\ell d} +\text{strictly lower degree terms},
      \]
     and, since $c_\ell(u+t) \alpha_{r}^\ell \neq 0$, it follows
     that $\varphi D(v)$ is of multidegree $\ell d$. On the other hand, 
      \begin{align*}
         D\bigl(\varphi(v)\bigr)= D(f_2)&= \sum_{r}\delta(\alpha_r)x^r+
        \sum_{r}\ \sum_{j=1}^n i_jb_j(u,v,x_1,\dots, x_n)\alpha_rx^{r-e_j}\\
        &=\delta(\alpha_{d})x^{d}+ \text{strictly lower degree terms,} 
      \end{align*}
        where $\{e_j\}$ is the canonical basis of the lattice $\Z^n$
          and  $\delta\in\der\bigl(\K[u,v]\bigr)$ is
          restriction of $D$ to $\K[u,v]$. It follows  that
          $D\bigl(\varphi(v)\bigr)$ is at most of multidegree $d$.  
      
      From the equation $\varphi D(v)=D\varphi(v)$ we deduce $\ell\leq
      1$ and hence $\ell=1$.
      It follows that  $c_1(u+t)\alpha_{d}=\delta(\alpha_{d})$ and
      $\alpha_{d}$ is a Darboux polynomial for $\delta$, 
    and therefore it is a constant. But $c_1(u+t)\neq 0$, so
    $\alpha_{d}=0$ which is a contradiction. 
       
   Applying the same reasoning to $\varphi^{-1}\in \aut(D)$ we
      deduce that $(f_1,f_2)\in \aut(\delta)$, and it follows from  Theorem
      \ref{th:MePan} that  $(f_1,f_2)=\varphi|_{_{\K[u,v]}}=\Id_{\K[u,v]}$.     
   \end{proof}

\begin{pro}\label{pro_extension}
Let $\delta\in\der\bigl(\K[u,v]\bigr)$ be a simple derivation and let
$I=\{a_1,\ldots, a_n\}\subset  \K[u,v]$ be a  linearly independent
subset.  If the linear span of $I$ is such that 
$\langle I\rangle_\kk\cap \operatorname{Im}(\delta)=\{0\}$, then the derivation
$D_I\in\der\bigl(\Bbbk[u,v,x_1,\ldots,x_n]\bigr) $ obtained by extending
$\delta$ as $D_I(x_j)=a_j \mbox{ for }j=1,\ldots,n$ is simple.
\end{pro}

\begin{proof}
We proceed by
induction on $n=\# I$; denote $D_n=D_I$. If $n=0$, then $D_0=\delta$ is simple by
hypothesis. 
Suppose now that $D_I$ is a simple derivation of
$\kk[u,v,x_1,\dots,x_n]$ for $\# I\leq n$, and consider $I=\{a_1,\ldots,
a_{n+1}\}$ as in the hypothesis. Then $D_{I}$ restricts to a simple
derivation $D_n\in \der\bigl(\K[u,v,x_1,\dots,x_n]\bigr)$ by hypothesis. 

By Shamsuddin's
criterion (see Proposition \ref{prop:sham}) $D_{n+1}$ is not simple if and only if there exists $f\in
\K[u,v,x_1,\ldots,x_n]$ such that $D_n(f)=a_{n+1}$ . If $f\in
\K[u,v,x_1,\ldots,x_n]$ is such that
$D_n(f)=a_{n+1}$, write  
$f=\sum_{r\in \Z_{\geq 0}^n}\alpha_rx^r$, with $\alpha_r\in \K[u,v]$,
and  let  $d$ be the multidegree of $f\in \K[u,v][x_1,\dots,x_n]$ for the
lexicographic order. Then \begin{equation}\label{eqn:DISimple}
  \begin{split}
    a_{n+1}= D_n(f)&=\sum_{r}\delta(\alpha_r)x^r+
        \sum_{r}\ \alpha_r\sum_{j=1}^n i_ja_jx^{r-e_j}\\
        &=\delta(\alpha_{d})x^{d}+ \text{strictly lower degree terms}. 
      \end{split}  
      \end{equation}

By considering the term of degree $d$ in Equation
\eqref{eqn:DISimple}, we deduce that  $\delta(\alpha_{d})=a_{n+1}$ if
$d=0$ or $\delta(\alpha_{d})=0$ otherwise. In the first case, we
deduce that $a_{n+1}$ belongs to $\langle I\rangle_\kk\cap \operatorname{Im}(\delta)$
and $a_{n+1}=0$ which is a contradiction.  

If $\delta(\alpha_{d})=0$, then $\alpha_{d}\in \kK$ because $\delta$ is simple. 
Consider $j_0=\max\{j\mathrel{ : } (d)_j\neq 0\}$, and let $d'=d
-e_{j_0}$. By definition of $j_0$, for all $j=1,2,\ldots,n$ and for
all multi-indexes  $r$ such that $0\leq r<d$ we have $d'\neq
r-e_j$. We deduce that the term of degree $d'$ in Equation
\eqref{eqn:DISimple} is 
\[
\delta(\alpha_{d'})+\alpha_{d}i_{j_0}a_{j_0}
\]
and this term is equal to $a_{n+1}$ if $d'=0$ or $0$ otherwise. In
both cases, we have a contradiction with the hypothesis $\langle
I\rangle_\kk\cap \operatorname{Im}(\delta)=\{0\}$.  
\end{proof}

The following example exhibits a derivation $\delta\in\der\bigl(\K[u,v]\bigr)$
such that $\delta$ admits linearly independent subsets $I$ as in the
hypothesis of Proposition \ref{pro_extension}, with arbitrary
cardinal. Moreover, for the family of simple derivations that we
produce the bound given in 
Theorem \ref{thm:dimaut0} is reached.

\begin{exa}\label{exa_maximal}
 Consider the derivation
 $\delta=\frac{\scriptstyle\partial}{\scriptstyle\partial
   u}+(1+uv)\frac{\scriptstyle\partial}{\scriptstyle\partial v}\in
 \der\bigl(\K[u,v]\bigr)$ --- notice that $\delta$ it is simple by
 Shamsuddin's Criterion. 

 Let
  us show that $\operatorname{Im}(\delta)\cap \kk[v]=\kk$. Assume that there exists
  $f=\sum_{i,j}\alpha_{i,j}u^iv^j\in\kk[u,v]$ such that $\delta(f)\in\kk[v]$. 
  By a direct computation, we have that:
  \[
  \delta(f)
  =\sum_{i,j}\alpha_{i,j}(iu^{i-1}v^{j}+ju^{i}v^{j-1}+ju^{i+1}v^{j}).
\]

 Let $(i_0,j_0)$ be the multidegree of $f$ for lexicographic order with $v\geq u$. If $j_0\neq 0$, then $\delta(p)$ is of multidegree $(i_0+1,j_0)$, and we cannot have $\delta(p)\in\kk[v]$.
 If $j_0=0$ then $f\in\kk[u]$ and $\delta (f)\in\kk[v]$ implies
 $\delta(f)\in\kk$ and our assertion follows.

Notice in particular that there exist linearly independent subsets $I$ as in the
 hypothesis of  Proposition \ref{pro_extension}, of arbitrary finite
 cardinal.
   \end{exa}

  \begin{cor}\label{cor:boundsharp}
    For every $n\geq 2$, there exists a simple derivation $D\in
    \operatorname{Der}\bigl(\pol\bigr)$ such that $\dim \aut(D)^0=n-2$.
\end{cor}
\begin{proof}
 The case  $n=2$ is the content of Theorem \ref{th:MePan}. If $n>2$,
  consider $\delta$ as in Example \ref{exa_maximal} and let
  $I=\{b_1,\dots ,b_n\} \subset \K[v]$  a linearly independent subset of cardinal $n$,
  such that $\langle I\rangle_\K\cap
  \operatorname{Im}(\delta)=\{0\}$. Then the 
derivation 
  $D_I=\delta+\sum_i
    b_i(v)\frac{\scriptstyle\partial}{\scriptstyle\partial x_i}$.
    is simple by Proposition \ref{pro_extension}.  Moreover, since  $D_I$ commutes with
    $\frac{\scriptstyle\partial}{\scriptstyle\partial x_1},\ldots,
    \frac{\scriptstyle\partial}{\scriptstyle\partial x_n}$ for $n\geq
    1$, it follows that
$      \aut(D_I)^0$ contains the subgroup $\G_a^n$    of translations on the $x$
coordinates. Hence $\aut(D_I)^0=\G_a^n$ by Theorem \ref{thm:dimaut0},  
      and the result follows.  
\end{proof}

    \section{Derivations invariant under the action of a group of translations}
\label{sec:translations}

Let $D$ be a simple derivation such that $\aut(D)$ contains a non
trivial subgroup of translations.  In this section we give some
insight on how to exploit this fact in order to describe $\aut(D)$.

\subsection{Isotropy groups with non trivial subgroups of translations}

 \begin{lem}\label{lem:forAn}
 	Let $H\subset \aut(\A^n)$, $H\cong G_a^s$, $1\leq s\leq n-1$, be a
 	subgroup of automorphisms acting 
 	in a globally trivial way, and consider an equivariant
        isomorphism  $\A^n\cong \A^s\times V$ as in Definition
        \ref{defi:globtriv}. Then the normalizer of $H$ in
        $\aut(\A^n)$ has the form
 	\[
 	N_{\aut(\A^n)}(H)=\bigl\{(x,v)\mapsto \bigl(A x + g_{1}(v), g_2(v)\bigr)
 	\mathrel{:} A\in \operatorname{GL}_s(\Bbbk), g_1:V\to\A^s, g_2\in \aut(V)\bigr\}.
      \]
 \end{lem}
 \begin{proof}
We describe an  automorphism $f\in\aut(\A^n)\cong \aut(\A^s\times V)$ as a
pair of
morphisms $(f_1,f_2)$, with $f_1:\A^s\times V\to \A^s$,
$f_2:\A^s\times V\to V$. Then
  $H=\bigl\{(t_a,\operatorname{Id}_V)\mathrel{:} a\in
 	\Bbbk^s\bigr\}$, and if $f=(f_1,f_2) \in N_{\aut(\A^n)}(H)$,
        we have that  $\sigma_f: H\to H$,
 	$\sigma_f(t_a,\operatorname{Id}_V)=
        f(t_a,\operatorname{Id}_V)f^{-1}$         is  a   
 	morphism of algebraic groups. It follows that  there exists
        $A_f\in GL_s(\K)$ such that $\sigma_f(t_a,\operatorname{Id}_V)=(t_{A_fa}, \operatorname{Id}_V)$ for all
        $a\in \K^s$ --- recall that
 	$\operatorname{char}(\Bbbk)=0$.

  From the equality
 	\[
 	f(t_a,\operatorname{Id}_V)=(t_{A_fa},
        \operatorname{Id}_V)f:\A^s\times V\to \A^s\times V
 	\]
 	we deduce that
 	\[
 	f_1\bigl(x+a, v)= f_1\bigl(x, v\bigr)  +A_fa \quad , \quad f_2(x+a,
 	v)= f_2\bigl(x, v\bigr) \text{ for all } a\in \Bbbk^s\,, v\in
        V.
 	\]

        Let  $f_1(x,v)=\bigl( f_{11}(x,v),\dots f_{1s}(x,v)\bigr)$ and consider the
 	maps  $f_{1j}$ as polynomials in $\Bbbk[V][x]$. Then
        $f_{ij}(x+a,v)=f_{ij}(x,v)+ (A_fa)_j$ and an easy
 	calculation  on the coefficients shows that
 	\[
          f_{1}(x,v)=  A_fx + g_1(v)\,, g_1:V\to \Bbbk^s
 	\]
 	
        On the other hand, since $f_2:\A^s\times V\to V $ is  a
        $\G_a^s$-invariant morphism and that $p_2:\A^s\times V\to V$
        is the geometric quotient, it follows that there exists
        $g_2:V\to V$ such that $f_2(x,v)=g_2\circ p_2(x,v)=g_2(v)$.
          
                        Finally, applying the same reasoning to the
                        inverse $f^{-1}\in   N_{\aut(\A^n)}(H)$, we deduce that  $g_2\in\aut(V)$.
 \end{proof}

 \begin{pro}\label{pro:desctrans}
 	Let $D$ be a simple derivation and $H\subset \aut(D)$ a normal
 	 subgroup, $H\cong \G_a^s$, $s\geq 1$, that acts in a globally
         trivial way.  Consider an equivariant
        isomorphism  $\A^n\cong \A^s\times V$ as in Definition
        \ref{defi:globtriv} and let 
        $\overline{D}=D|_{_{\K[V]}}:\K[V]\to\K[V]$ (see Proposition
        \ref{prop:tras0}). If we identify  $\aut(D)$ as a subgroup of
        $\aut(\A^s\times V)$ and $\aut(\overline{D})$ as a subgroup of
        $\aut(V)$, then  
 	\[
 	\aut(D)\subset \bigl\{ \bigl(Ax+g_1(v), g_2(v)\bigr)\mathrel{:}
 	A\in\operatorname{GL}_s(\Bbbk)\,,\ g_1:V\to\A^s\,,\ g_2\in \aut(\overline{D})\bigr\}.
 	\]
 \end{pro}
 \begin{proof}
 	   By Lemma 
        \ref{lem:forAn},  follows  that
 	\[
 	\aut(D)\subset \bigl\{ \bigl(Ax+g_1(v), g_2(v)\bigr)\mathrel{:}
 	A\in\operatorname{GL}_s(\Bbbk)\,,\ g_1:V\to\A^s\,,\ g_2\in \aut(V)\bigr\},
 	\]
 	so it remains to prove that if $f=\bigl(Ax+g_1(v),
        g_2(v)\bigr) \in \aut(D)$, then $g_2\in \aut(\overline{D})$
        --- that is, after identification of $\aut(\A^n)$ with
        $\aut\bigl(\pol\bigr)$, that $Dg_2(p)=g_2D(p)$ for all
        $p\in\K[V]$. But by  definition, we have that
        \[
\begin{split}
         \overline{D}g_2(p)=&D
         \bigl(p\bigl(g_2(v)
         \bigr)\bigr)=D
         \bigl(p\bigl(f_2(x,v)
         \bigr)\bigr)=D
         \bigl(p\bigl(f(x,v)
         \bigr)\bigr)=Df(p) \\
         g_2 \overline{D}(p)= &fD(p),
\end{split}\]
where we consider $p \in\K[V]\subset\K[V][x]$, and the result follows.
 \end{proof}
 
 \begin{rem}\label{rem:homind}
  If $X,Y$ are affine algebraic varieties, then 
  $\operatorname{Hom}(X,Y)$ inherits a structure of ind-variety (see
  for example
  \cite[Lemma 3.1.4]{kn:furterkraft}). In the notations of Proposition
  \ref{pro:desctrans},  we
  identify  $\aut(\A^n)=\aut(\A^s\times V)$ as a subset of $\operatorname{Hom}(\A^s\times V, \A^s)\times
  \operatorname{Hom}(\A^s\times V, V)$, and restrict the projection
  over the second  
  coordinate  to $ \aut(D)$. Then $p_2
 \bigl( \aut(D)\bigr)$ identifies with a subgroup of $\aut(\overline{D})$,
 in such a way that the corresponding map  $\varphi:\aut(D)\to
 	\aut(\overline{D})$,  is a morphism of ind-groups.
        It follows that
 	\[
 	\aut(D)^0\subset \bigl\{ \bigl(Ax+g_1(v), g_2(v)\bigr)\mathrel{:}
 	A\in\operatorname{GL}_s(\Bbbk)\,,\ g_1:V\to \A^s\,,\ g_2\in \aut(\overline{D})^0\bigr\}.
 	\]
 \end{rem}

 \begin{thm}\label{th:AutAlg}
 	Let $D$ be a simple derivation such that $\aut(D)$ contains a
 	non trivial normal subgroup $H\cong \A^s$, $s\geq 1$,  such that
$H$ acts on a globally trivial way. Consider  an equivariant
        isomorphism  $\A^n\cong \A^s\times V$ as in Proposition
        \ref{pro:desctrans} and let
        $\overline{D}=D|_{_{\K[V]}}$. Assume moreover that the
        restriction $\overline D:\Bbbk[V]\to \Bbbk[V]$ is such that
 	$\aut(\overline{D})$ is algebraic. 
 	
 	Then $\aut(D)$ is algebraic; in particular,  $\aut(D)=\aut(D)^0$.
 \end{thm}
 \begin{proof}
 	By Proposition \ref{pro:desctrans}, if $f\in \aut(D)$ then
 	$    f(x,v)=\bigl(Ax+g_1(v), g_2(v)\bigr)$, with
 	$A\in\operatorname{GL}_s(\Bbbk)$ and $g_2\in \aut(\overline{D})$.
 	By induction on the number of compositions, we deduce that for $\ell\in\N$
 	\[
 	f^\ell(x,v)=\Bigl( A^\ell x+ \sum_{i=0}^{\ell-1}
          A^{\ell-i}g_1\bigl(g_2^{i}(v)\bigr)\,, \,  g_2^\ell(v)\Bigr).
 	\]

Consider the equivariant isomorphism  $\psi:\A^n\to \A^s\times V$ and
the isomorphism of ind-groups $\aut(\A^n)\to \aut(\A^s\times V)$ given by conjugation by $\psi$. Then 
$ \psi^{-1}(\operatorname{Id}_{\A^s}, g_2)
\psi: \A^n\to \A^n$ is an algebraic
automorphism, and it follows that, under identification
by $\psi$,  the family $\bigl\{\bigl(A^\ell x+ \sum_{i=0}^{\ell-1}
          A^{\ell-i}g_1(g_2^{i}(v)), g_2(v)\bigr):\A^n\to \A^n\mathrel{:}
          \ell\geq 0\bigr\}$ has bounded degree,
          and therefore $f$ is algebraic. 
      \end{proof}

      \begin{cor}\label{cor:castrasn-2}
 	Let $D\in\der\bigl(\pol\bigr)$  such that
        $\aut(D)^0=\G_a^{n-2}$, $n\geq 2$, and assume moreover that 
$\aut(D)^0$ acts by translations. Then $\aut(D)=\aut(D)^0$.
\end{cor}
\begin{proof}
This is a direct application of Theorem \ref{th:AutAlg}, together with
the fact that a simple derivation of $\K[u,v]$ has trivial isotropy
(Theorem \ref{th:MePan}).
 \end{proof}

\begin{exa}
Notice that in particular,  Corollary \ref{cor:castrasn-2} shows that
the derivations $D_I$  exhibited   in  Example \ref{exa_maximal}  are
such that $\aut(D_I)$ is algebraic (see also  Corollary
\ref{exa_maximal}).    \end{exa}

 \subsection{Simple derivations of $\pol$ in small dimension  and their
        automorphisms}\ %

In this section we study the automorphisms group of a simple derivation
of $\pol$, with $n=3,4$, the cases $n=1,2$ being well known:

\beginenum

\item A derivation 
$D\in\der\bigl(\K[x]\bigr)$ is simple if and only if it is locally
nilpotent, and in this case $\aut(D)=\{e^{tD}\mathrel{:} t\in\K\}$.

\item Any simple derivation of
$\K[x,y]$ has trivial isotropy (see Theorem \ref{th:MePan}). 
\end{enumerate}

 \subsection*{Simple derivations of $\K[x_1,x_2,x_3]$}\ %

Let $D$ be a simple derivation of
$\Bbbk[x_1,x_2,x_3]$. Our objective is to show that $\aut(D)$ is
either isomorphic to $\G_a$ or possible a countable discrete group:
More precisely,  $\aut(D)^0$ is unipotent
and of dimension $0$ or $1$  by  Theorem
  \ref{thm:dimaut0}; the following theorem shows that in this last case, $\aut(D)=\aut(D)^0$, its action being by translations.

\begin{thm}\label{thm:autDim3}
	Let $D\in\der\bigl(\Bbbk[x_1,x_2,x_3]\bigr)$ be a simple derivation such that $\dim \aut(D)^0=1$. Then
	$\aut(D)=\aut(D)^0$. Moreover, there exist coordinates such that
	\begin{equation}
		\label{eqn:A3}
		\begin{split}
			\aut (D) &=\bigl\{ (x_1+a, x_2,x_3)\mathrel{:} a\in\Bbbk\}\\
			D &= \sum a_i(x_2,x_3)\frac{\partial}{\partial x_i}
		\end{split}
	\end{equation}
\end{thm}

\begin{proof}
  If  $u\in \aut(D)^0$ is a non trivial automorphism, then by
  Proposition \ref{prop:identity}, $u$ has no fixed point but, by a
  result of Kaliman (see \cite{kn:kaliman}),  a non trivial unipotent
  automorphism $u\in \aut(\A^3)$ without fixed point is conjugated to
  a  translation.  It follows from Proposition
\ref{pro:desctrans}  that there exist coordinates as in Equation
\eqref{eqn:A3} but for $\aut(D)^0$. Hence, it remains to prove that
$\aut(D)=\aut(D)^0$; this is the content of Corollary \ref{cor:castrasn-2}.
\end{proof}

  \subsection*{A family of simple derivations of $\Bbbk[x_1,x_2,x_3,x_4]$}\ %

Next, we consider the four dimensional affine space and
describe the automorphisms group of a simple 
derivation $D\in \der \bigl(\Bbbk[x_1,x_2,x_3,x_4]\bigr)$ that admits
a linear coordinate.

\begin{lem} \label{lem:deltatrans}Let $D\in \der\bigl(\pol\bigr)$ be a
  simple derivation that admits a linear coordinate  $s$. Let
  $\phi\in\aut(D)^0$ and $\Delta$ be a locally nilpotent derivation such
  that  $\phi=e^\Delta$  (see Remark \ref{rem:GaandLND}). Then $\Delta(s)\in\K$. In
  particular, then either $\phi(s)=s$ or $\phi$ is a translation along the
  direction of $s$. 
\end{lem}
\begin{proof}
	Since $\phi \in \aut(D)$, then   $[D,\Delta]=0$ and therefore
        $D\bigl(\Delta(s)\bigr)=0$, so $\Delta(s)\in\K$.  
	
	If $\Delta(s)=0$, then $\phi(s)=s$.
	On the other hand, if $\Delta(s)=c\neq 0$, then
        $\Delta(c^{-1}s)=1$, i.e. $c^{-1}s$ is a slice for
        $\Delta$. Hence,
        $\pol=\ker(\Delta)[c^{-1}s]$. Since $c^{-1}s$ is a
        coordinate we may, up to conjugation, assume $x_1=c^{-1}s$. Thus
        $\ker(\Delta)=\K[x_2,\ldots,x_{n}]$,
        i.e. $\phi$ is 
        translation along the $s$ coordinate.
\end{proof}

Now we specialize Lemma \ref{lem:deltatrans} to the case where
$n=4$. We begin by a useful result.  

\begin{lem}\label{lem:TransDim4}
	Let $D\in \der\bigl(\K[x_1,x_2,x_3,x_4]\bigr)$ be a simple
        derivation that admits a linear coordinate. Then every
        nontrivial element of $\aut(D)^0$ is 
        conjugate to a translation. 
\end{lem}

\begin{proof}
	Let $\phi\in\aut(D)^0$ and  let $\Delta$ be a locally nilpotent
        derivation such that  $\phi=e^\Delta$. By Lemma
        \ref{lem:deltatrans}, it suffices to prove that if $s$ is a
        linear coordinate for $D$ such that  $\Delta(s)=0$, then
        $\phi$ is a translation. By Proposition \ref{prop_locally} we know that
	$\G_a=\{e^{t\Delta}; t\in\K\}$ acts over $\A^4$ in a locally trivial, and
        in particular a proper,
        way. Then,  the result follows from
        \cite[Thm. 01]{kn:kaliman2}.   
\end{proof}

\begin{thm}\label{thm:dim4}
	Let $D\in \der\bigl(\K[x_1,x_2,x_3,x_4]\bigr)$ be a simple
        derivation that admits a
        linear coordinate.  If $\dim\aut(D)^0>0$, then $\aut(D)^0$
        acts by translations. Moreover, if $\dim\aut(D)^0=2$, then
        $\aut(D)$ is algebraic. 
\end{thm}
\begin{proof}
    Recall that, by Theorem \ref{thm:dimaut0}, $\aut(D)^0$ is a unipotent
    group of dimension $1$ or $2$. In the case of dimension 1	the
    theorem is a direct consequence of Lemma \ref{lem:TransDim4}.  
	
	Now, we suppose that $\dim \aut(D)^0=2$. It is well known that in this
        case $\aut(D)^0=U_1\times U_2$, $U_i\cong \G_a$, --- recall that
        $\operatorname{char}\K=0$. Consider two generators $u_1,u_2$ of
        $U_1$ and $U_2$ respectively. Again by Lemma \ref{lem:TransDim4} we may suppose that
        $u_1$ is a translation with respect to $x_1$; therefore,
        we deduce from Lemma \ref{lem:forAn} that $u_2=\bigl(ax_1+g_1(x_2,x_3,x_4),
        g_2(x_2,x_3,x_4),g_3(x_2,x_3,x_4),g_4(x_2,x_3,x_4)\bigr)$;
        moreover, 
        since $u_1$ and $u_2$ commute,   we deduce that $a=1$. Then the
       geometric  quotient  $q_1:\A^4\to X_1:=\A^4/\overline{\langle U_1\rangle}$ exists and
        is isomorphic to $\A^3$, and $D$ induces a simple derivation
        $\overline{D}\in \der\bigl(\K[X_1]\bigr)=\der\bigl(\K[x_2,x_3,x_4]\bigr)$. As $u_1$ and $u_2$
        commute, we deduce that $U_2$ acts on  $X_1$, in such a way that
        $U_2\subset \aut(\overline{D})$. It follows from  Theorem
        \ref{thm:autDim3} that either the action of $U_2$ over $X_1$ is trivial
        or given by translations.

        If $U_2$ acts trivially, it follows that $u_2=
        \bigl(x_1+g_1(x_2,x_3,x_4), x_2,x_3,x_4\bigr)$. But if
        $(p_2,p_3,p_4)\in\mathcal{V}(g_1)\subset \A^3$, it follows that
        $(0,p_2,p_3,p_4)$ is a fixed point of $u_2$, and we obtain a
        contradiction. Hence $U_2$ acts by translations over $X_1$.
        Changing coordinates in $X_1=\A^3$, we can 
        assume that $u_2=\bigl( x_1+ w(x_2,x_3,x_4), x_2+1,
        x_3,x_4)$. If $\Delta$ is the locally nilpotent derivation such that $u_2=e^{\Delta}$, then $\Delta=a\partial/\partial x_1+\partial/\partial x_2$, with $a\in\K[x_2,x_3,x_4]$. Indeed, if $w$ has degree $\ell$ with respect to $x_2$, then $a$ is determined by the equation
       \begin{eqnarray*}w&=&u_2(x_1)-x_1\\
         &=&\Delta(x_1)+\frac{1}{2}\Delta^2(x_1)+\cdots+\frac{1}{\ell!}\Delta^{\ell}(x_1)\\
                         &=&a+\frac{1}{2}\frac{\partial}{\partial x_2}(a)+\cdots+\frac{1}{\ell!}\frac{\partial^{\ell-1}}{\partial x_2^{\ell-1}}(a).
       \end{eqnarray*}
 Let $A\in\K[x_2,x_3,x_4]$ be a polynomial such that $\partial/\partial x_2 (A)=a$, and consider the coordinates
 $z,x_2,x_3,x_4$, where $z= x_1-A+x_2$.  Then  $u_1(z)=z+1$, and since $\Delta(z)=1$  then $u_2(z)=z+1$. We conclude that in these new coordinates
 $u_1=(z+1,x_2,x_3,x_4)$ and $u_2=(z+1,x_2+1,x_3,x_4)$, and  therefore
 $\aut(G)^0$ is included in the group of
translations.

In order to finish the proof, we apply Corollary  \ref{cor:castrasn-2}.
\end{proof}

  We finish by showing that the absence of linear
   coordinates is not an obstruction for a simple derivation to have
   algebraic automorphisms group.

\begin{exa}
In \cite{kn:Jor},  the author shows that the derivations
$D_n\in \der\bigl(\pol\bigr)$ given by
\[
  D_n=(1-x_1x_2)\frac{\partial}{\partial{x_1}}+x_1^3\frac{\partial}{\partial{x_2}}+\sum_{i=3}^nx_{i-1}\frac{\partial}{\partial{x_i}}
\]	
are simple for all $n\geq 2$. In \cite{Ya}, the author shows that if
$n\geq 3$, then 
$\aut(D_n)\cong \G_a$, acting by translations in the last coordinate
--- recall that $\aut(D_2)=\{\operatorname{Id}\}$. We affirm that
$\operatorname{Im}(D_n)\cap \Bbbk=\{0\}$ --- in particular, the family $D_n$,
$n\geq 2$, 
gives  examples of simple derivations without linear coordinates such
that their automorphism group is algebraic.

Indeed, $\operatorname{Im}(D_n)$ is the linear span of $G=\{D_n(x^d)\mathrel{:} d\in\N^n\}$. 

Let $(e_i)_{1\leq i\leq n}$ be the canonical basis of $\Z^n$, a direct
computation shows that: $D_n(x^{e_1})=1-x_1x_2$ is the only polynomial
which contains monomials $1$ and $x_1,x_2$ with non zero coefficient
in $G$ and so $1\notin\operatorname{Im}(D_n)$. 
\end{exa}

\end{document}